\documentclass[12pt,oneside,reqno]{amsart}
\usepackage{amsmath}
\usepackage{amsfonts}
\usepackage{amssymb}
\usepackage{graphicx}
\usepackage{enumerate}
\usepackage[mathscr]{eucal}

\theoremstyle{plain}

\newtheorem*{corollary}{Corollary}

\newtheorem{fact}{Fact}

\newtheorem*{lemma}{Lemma}

\newtheorem{proposition}{Proposition}
\theoremstyle{remark}
\newtheorem*{remark}{Remark}

\theoremstyle{plain}
\newtheorem*{observation}{Observation}
\theoremstyle{definition}

\usepackage[height=210mm,width=158mm]{geometry}                       
\begin{document}
\title[
Log-concavity of rows of 
 Pascal type triangles]{
Log-concavity of rows of 
 Pascal type triangles}

\author{Stephan Foldes}
\address{Stephan Foldes\newline%
\indent University of Miskolc,  \newline%
\indent 3515 Miskolc-Egyetemv\'aros, Hungary }
\email {foldes.istvan@uni-miskolc.hu}%

\author{L\'aszl\'o Major}
\address{L\'aszl\'o Major\newline%
\indent University of Turku,  \newline%
\indent Faculty of Science and Engineering    \newline%
\indent  20500 Turku, Vesilinnantie 5, Finland }
\email{laszlo.major@utu.fi}%

\hspace{-4mm} \date{6 March 2019}

\begin{abstract}
Menon's proof of the preservation of log-concavity of sequences under convolution becomes simpler when adapted to 2-sided infinite sequences. Under assumption of log-concavity of two 2-sided infinite sequences, the existence of the convolution is characterised by a convergence criterion. Preservation of log-concavity under convolution yields another method of establishing the log-concavity of rows of certain Pascal type triangles. This includes in particular the log-concavity of rows of a weighted Delannoy
triangle. The method is also compared with known techniques of proving log-concavity.   
 \end{abstract}

\maketitle
\section{Introduction}
For integer indexed 2-sided infinite sequences $(a_{n})_{n\in \mathbb{Z}}$ of non-negative real numbers we consider the \textit{log-concavity} property without internal zeros, which means that for all integers $n$ and for all positive integers $p,q$ we have  $a_n a_{n+p+q}\leq a_{n+p} a_{n+q}$. Unimodality of combinatorial sequences can be proved by establishing the stronger property of log-concavity. The literature often refers back to the seminal papers by Stanley \cite{sta} and Brenti \cite{bre}.

In recent years there has been significant progress in the study of log-concavity and unimodality properties of sequences appearing as rows and other transversals or rays in combinatorial (arithmetic) triangles generalizing the triangle of Pascal. The articles of Belbachir and Szalay \cite{besz}, of Su and Wang \cite{suw} and of Yaming \cite{yam} deal with the classical Pascal's triangle, considering several types of transversals or rays. The approach to generalizing Pascal's triangle which consists of considering coefficient sequences of powers $p^n$ of   a given polynomial $p$ as successive rows of an array - originating with Euler - is central in the articles of Su and Zhang \cite{suz} and of Ahmia and Belbachir \cite{ab}. Both of these articles study the log-concavity of the sequences of the coefficients of the polynomials $p^n$. If the polynomial $p$ involved has several variables, then instead of Pascal type triangles, Pascal type pyramids arise in which the various transversals and rays have been studied by Belbachir, Bencherif and Szalay \cite{bbsz} and by Su and Wang \cite{suwa}. Pascal's triangle can also be generalized using recursion formulas where each entry is computed as a weighted sum of some of the entries above it in the triangle. Within this framework the log-concavity property in weighted Delannoy triangles has been studied by Mu and Zheng \cite{muz}. The articles of Lin and Zeng \cite{lin} and of Zhu \cite{zhu}  deal with variable weights in the recursion formulas.

In Section \ref{ptta} of this note we return to an older construction of convolution triangles, where the main focus was not log-concavity, introduced in the 70's  (see e.g. Hoggat and Bicknell \cite{hog1}). 

The following Facts, which do not appear to have been stated explicitly in the related literature, are easily verified. Also Proposition \ref{f5} and the Observation in the next section seem to be new, but their proofs are straightforward and will thus be omitted.

\begin{fact}
If $(a_n)_{n\in \mathbb{Z}}$ and $(b_n)_{n\in \mathbb{Z}}$ are log-concave then the term-wise product $(a_nb_n)$
is also log-concave.
\end{fact}
As well known, every log-concave sequence is unimodal. Indeed the property of log-concavity itself can be characterized by the unimodality of certain derived sequences:
\begin{fact}\label{f3}
The sequence $(a_n)_{n\in \mathbb{Z}}$ is  log-concave if and only if for  each  fixed $p$ the skew self-product $(a_n a_{p-n})_{n\in \mathbb{Z}}$  is unimodal.
\end{fact}
\begin{remark}
The modus of  $(a_n a_{p-n})$ will be reached at $\lfloor p/2 \rfloor$ (also at $\lceil p/2\rceil$), the floor and ceiling coinciding for even $p$.
\end{remark}
\begin{fact}\label{f4}The following are equivalent for any non-null log-concave sequence $(a_n)_{n\in \mathbb{Z}}$:
\begin{enumerate}[(i)]
\item the sequence has a finite sum,
\item the terms of the sequence tend to $0$ as $n$ tends to infinity or minus infinity,
\item the indices where the sequence has maximal values form a non-empty finite interval.
\end{enumerate}
\end{fact}

The convolution of two sequences  $a=(a_n)_{n\in \mathbb{Z}}$ and $b=(b_n)_{n\in \mathbb{Z}}$ of non-negative numbers, denoted   $a*b$, 
is the sequence  $(s_n)_{n\in \mathbb{Z}}$  with possibly infinite terms $s_n \in [0,\infty]$, 
where   $s_n$   is the possibly infinite (divergent) sum $$\sum_{m\in \mathbb{Z}}a_m\cdot b_{m-n}=\sum_{m\in \mathbb{Z}}a_{n-m}\cdot b_m $$
\begin{proposition}\label{f5}
Every term of the convolution of log-concave sequences $(a_n)_{n\in \mathbb{Z}}$ and $(b_n)_{n\in \mathbb{Z}}$ is finite if and only if for all $p$
the terms of each skew product $(a_nb_{p-n})$ tend to $0$ as $n$ tends to infinity or minus infinity.
\end{proposition}

\section{Preservation of log-concavity under convolution}

Certain properties of a sequence $(a_{n})_{n\in \mathbb{Z}}$, not necessarily log-concave, are conveniently discussed in terms of the associated formal Laurent series $\Sigma  a_nX^n$, the product of series $(\Sigma  a_nX^n)(\Sigma b_nX^n)$ corresponding to the \textit{convolution sequence} $(a_n)\ast (b_n)$ if it exists. For repeated convolutions of the sequence $(a_{n})_{n\in \mathbb{Z}}$ with itself ($k$-fold convolution with $k$ factors) the exponential notation $(a_n)^{*k}$ is used, where $(a_n)^{*1}=(a_n)$ and $(a_n)^{*0}=(\ldots,0,0,1,0,0,\ldots)$. An important class of Laurent series is that of  power series, where $a_n=0$ for all negative indices.

It is known that the coefficient sequence of the product of two formal power series with log-concave non-negative coefficient sequences is also log-concave. This result was clearly stated and proved by Menon \cite{men}, and can also be deduced from the theory of total positivity of Karlin \cite{kar} developed in a larger context. In fact Karlin already  deals with 2-sided infinite coefficient sequences, in other words, formal Laurent series. The product of two formal Laurent series with non-negative coefficients may not exist, but if it does, then log-concavity of the factors implies log-concavity of the product. This fact is somewhat simpler to verify for Laurent series than for power series (as Menon did), because Laurent series can be multiplied with any positive or negative power of the indeterminate $X$, a reversible operation which obviously preserves log-concavity of the coefficient sequence. Thus we need to verify only that in the product
$$(\Sigma  a_nX^n)(\Sigma b_nX^n)=\Sigma c_nX^n$$
we have
\begin{equation}\label{eq1}
c_0^2 \geq (c_{-1})(c_1).
\end{equation}
The coefficients in question are infinite sums, assumed to be convergent,
\begin{equation*}
\begin{split}
c_0=& \Sigma a_k b_{-k}\\
c_{-1}=& \Sigma a_k b_{-k-1}\\
c_1 = &\Sigma a_{k} b_{-k+1}
\end{split}
\end{equation*}

For showing that (\ref{eq1}) holds, the following general observation will be useful, applying to any situation when we need to compare two sums of non-negative real numbers indexed by the same index set, but term-by-term comparability for each index is not available. Recall that the sum of a countable family of non-negative real numbers is always a well-defined element of the extended real half-line $[0,\infty].$

\begin{observation}
Let $(p_i)$ and $(r_i)$ be finite or infinite families of non-negative real numbers indexed by the same set $I$, and let $\sigma$  be a bijection $I\rightarrow I$ (permutation of indices). If for each index $i$ the inequality $p_i + p_{\sigma(i)} \geq r_i + r_{\sigma(i)}$
holds, then $\Sigma p_i  \geq  \Sigma r_i$ even if $p_i \geq r_i $ fails for  some indices $i$.
\end{observation}

Applying the Observation with $I=\mathbb{Z}^2=\{(k,n): k,n\in \mathbb{Z}\}$, $p_{k,n}=a_ka_nb_{-k}b_{-n}$ and $r_{k,n}=a_ka_nb_{-k-1}b_{-n+1}$, to show that $c_0^2 \geq (c_{-1})(c_1)$ we need to verify that \begin{equation}
\sum_{(k,n)\in \mathbb{Z}^2}p_{k,n}\geq \sum_{(k,n)\in \mathbb{Z}^2}r_{k,n}
\end{equation}
For this, we define the permutation $\sigma$ of the index set $\mathbb{Z}^2$ by $\sigma(k,n)=(n-1,k+1).$  For each indexing pair $(k,n)$, we claim that $$p_{k,n}+p_{\sigma(k,n)}\geq r_{k,n}+r_{\sigma(k,n)}$$
This inequality is equivalent to the non-negativity of
\begin{equation*}
\begin{split}
a_{k}a_{n}&b_{-k}b_{-n}+a_{n-1}a_{k+1}b_{-n+1}b_{-k-1}   
-a_{k}a_{n}b_{-k-1}b_{-n+1}-a_{n-1}a_{k+1}b_{-n}b_{-k}  \\
= &\, a_{k}a_{n}(b_{-k}b_{-n}-b_{-k-1}b_{-n+1})-a_{n-1}a_{k+1}(b_{-n}b_{-k}-b_{-n+1}b_{-k-1})\\
= & \,(a_{k}a_{n}-a_{n-1}a_{k+1})(b_{-k}b_{-n}-b_{-k-1}b_{-n+1})
\end{split}
\end{equation*}
Simple case analysis shows that the two parentheses above are of the same sign, based on the log-concavity of $(a_n)$ and $(b_n)$. This argument, amounting to a version of Menon's proof for 1-sided infinite sequences \cite{men}, confirms the following:
\begin{proposition}[see e.g. Karlin\cite{kar}, Theorem 8.2]\label{p5}
If every term of the convolution of log-concave sequences $(a_n)_{n\in \mathbb{Z}}$ and $(b_n)_{n\in \mathbb{Z}}$ is finite, then the convolution $(a_n)\ast (b_n)$ is also log-concave.
\end{proposition}
Note that the dual property of log-convexity is not preserved under convolution (see Liu and Wang \cite{liu}).

\section{Pascal type triangular arrays}\label{ptta}
Several broad generalizations of Pascal's triangle were developed, in particular as an approach to log-concavity (or unimodality) of (finite) combinatorial sequences  (e.g. Stirling numbers \cite{mez}, $f$-vectors of polytopes). Kurtz's construction involves row to next row recursion using weighted summation of the two entries above a given entry \cite{kur}, a simpler instance of which was used to study $f$-vectors of cyclic polytopes \cite{maj}. Brenti's more recent weighted summation formula is more complex than Kurtz's, in that it involves more than a single preceding row, but it also has different assumptions about the weights \cite{bre1}. The triangular construction may also proceed from fixing the (infinite) right side of the triangle not necessarily consisting of 1's. This approach is implicit e.g.  in Hoggar \cite{hog}, the construction being in fact subsumed by the earlier {convolution triangle} construction of Hoggatt and Bicknell \cite{hog1}. In this section we establish the log-concavity of the rows of a large class of convolution triangles, a special case being the log-concavity of the rows of the Delannoy triangle.

A \textit{triangular array} is a doubly indexed family of numbers $T(n,k)_{n,k\in \mathbb{Z}}$, where $T(n,k)$ is $0$ unless $0\leq k \leq n$.
The $n^{th}$ \textit{row} of the array is the sequence $(T(n,k))_{k\in \mathbb{Z}}$. Obviously this row has at most $n+1$ non-null terms.
  For any two sequences $a=(a_n)$ and $q=(q_n)$ whose terms are zeros for $n<0$, the \textit{convolution array} is the triangular array given by $T(n,k)=(a\ast q^{\ast (n-k)})_{k}$ for $0\leq k\leq n$, $a=(a_n)$ being called \textit{initial side sequence} and $q=(q_n)$ the \textit{convolution multiplier sequence}.

\begin{lemma}\label{l1}Let $a=(a_n)$ and $q=(q_n)$ be  non-negative log-concave sequences whose terms are zeros for $n<0$. For $k\geq 1$ define the sequences 
\begin{equation*}
\begin{split}
b=& a\ast q^{\ast (k-1)}\\
c=& a\ast q^{\ast k}\\
d = & a\ast q^{\ast (k+1)}\\
\end{split}
\end{equation*} 
For all $n\geq 1$ \textit{the following inequality holds:} $$c_{n}^{2}\geq d_{n-1}\cdot b_{n+1}$$
\end{lemma}

\begin{proof}
Now 
\begin{equation}\label{sn2}
c_{n}^{2}=(b_{0}q_n+\cdots+b_{n}q_0)(b_{0}q_n+\cdots+b_{n}q_0)
\end{equation}
and
\begin{equation}\label{Sn}
d_{n-1}\cdot b_{n+1}=\bigg(\sum_{i=0}^{n-1}\sum_{j=0}^{n-i-1}b_{i}q_{j}q_{n-i-j-1}\bigg)\cdot b_{n+1}
\end{equation}
From (\ref{sn2}) we have

\begin{equation}\label{snn}
c_{n}^{2}\geq \sum _{i=0}^{n-1}\sum _{j=0}^{n-i-1}b_{i+j+1}b_{n-j}q_{n-i-j-1}q_{j}
\end{equation}
while because of the log-concavity of $b$, the product  $b_{i+j+1}b_{n-j}$ is always greater than or equal to $b_{i}b_{n+1}$ which proves the
Lemma due to the expression (\ref{Sn}). 
\end{proof}
Based on the Lemma   the following can be shown:
\begin{proposition}\label{thr}
For any two given log-concave sequences $a=(a_n)$ and $q=(q_n)$ whose terms are zeros for $n<0$, taken as initial side and convolution multiplier sequence respectively,  every row of the corresponding convolution array is log-concave. (The $n^{th}$ element in the $k^{th}$ row of the convolution array being given by $T(n,k)=(a\ast q^{\ast (n-1)})_k$ for $0\leq k \leq n$, $T(n,k)=0$ for $k\leq 0$ and $n\leq k$.)
\end{proposition}
As an application, consider the triangular array which for any fixed real $\beta,\gamma,\delta\geq 0$ is given by $T(0,0)=1$ and 
\begin{equation}\label{e6}
T(n,k)=\beta T(n-1,k-1)+\gamma T(n-1,k)+\delta T(n-2,k-1)
\end{equation}
In other words, each entry in the array other than $T(0,0)$ is calculated as a weighted sum of the two entries above and the third entry which is above the latter two. This is in fact a weighted generalization of the Delannoy triangle (counting weighted lattice paths as in Fray and Roselle \cite{frr} and Hetyei \cite{het}, but being less general than the generalized Delannoy numbers in Dziemia\'nczuk \cite{dzi}, differing also from Loehr and Savage \cite{los}).   The log-concavity of the rows of this generalized Delannoy triangle can be obtained as a corollary of Proposition \ref{thr} (or, in a different framework, derived from  Theorem 4.3 of Brenti    \cite{bre1}). More concretely, Proposition \ref{thr}  applies to the case where ${a}=(\ldots,0,0,1,\beta,\beta^2,\ldots)$ and ${q}$ itself is given as the convolution $${q}=(\ldots,0,0,1,\beta,\beta^2,\ldots)\ast(\ldots,0,0,\gamma,\delta,0,0,\ldots)$$
This shows that the following holds:
\begin{corollary}[derivable also from Brenti \cite{bre1}]
For any fixed real numbers $\beta,\gamma,\delta\geq 0$ the rows of the weighted Delannoy triangular array given by $T(0,0)=1$ and the recursion \emph{(\ref{e6})}  are log-concave. 
\end{corollary}
Obviously the $n^{th}$ row of this generalized Delannoy triangle is the sequence of coefficients of the monomials of total degree $n$ in the power series $$\sum_{m=0}^{\infty}(\beta x+\gamma y+\delta xy)^m$$
Pascal's triangle corresponds to the case $\beta=\gamma=1, \,\delta=0. $

A statement stronger than the above corollary was proved by Mu and Zheng in \cite{muz}, relying on the assumtion that $\beta=1$.
\section{Comparison of methods to prove log-concavity}
To establish the log-concavity of the rows of a triangular array the methods of Kurtz and Brenti assume certain weighted recursion relationships that yield each row from one, two or three previous rows, with different assumptions also about the weight sequences. On the other hand, Proposition \ref{thr} above assumes that the triangular array is a convolution array, generated from log-concave initial side and convolution multiplier sequences. 

Kurtz's method works to establish also the log-concavity of rows of certain convolution triangles, which was used in particular to prove log-concavity of $f$-vectors of ordinary polytopes (see \cite{maj}, where convolution triangles have eventually decreasing, not necessarily log-concave initial side sequences).
Proposition \ref{thr} works also to establish log-concavity of some triangular arrays where a row does not fully determine the next row (classical and generalized Delannoy triangles). 

The two approaches have a non-trivial overlap. For example, while an instance of Proposition \ref{thr} with convolution multiplier $q=(\ldots,0,0,1,1,1\ldots)$, as appearing implicitly in Hoggar \cite{hog}, provides a simple  approach to verify the log-concavity of $f$-vectors of cyclic polytopes, the first proof of this log-concavity property for the larger class of ordinary polytopes was based on an argument amounting to an instance of Kurtz's method \cite{maj}.


\begin{thebibliography}{9}                                                                                            
\vspace{1.2mm} 


\bibitem{ab} \textsc{M. Ahmia, H. Belbachir}, \textit{Unimodality polynomials and generalized Pascal triangles}, 	Algebra Discrete Math. 26.1 \vspace{0.7mm}(2018).

\bibitem{besz} \textsc{H. Belbachir, L. Szalay}, \textit{Unimodality of certain sequences connected with binomial coefficients}, 	J. Integer Seq. 10.2 \vspace{0.7mm}(2007) 07-2.

\bibitem{bbsz} \textsc{H. Belbachir, F. Bencherif, L. Szalay}, \textit{Unimodal rays in the regular and generalized Pascal pyramids}, Electron. J. Combin. 18.1 (2011)\vspace{0.7mm} P79.

\bibitem{bre} \textsc{F. Brenti}, \textit{Log-concave and unimodal sequences in algebra, combinatorics, and geometry: an update}, Contemporary Math. 178 (1994)  \vspace{0.7mm} 71-89

\bibitem{bre1}\textsc{F. Brenti}, \textit{Combinatorics and Total Positivity},  J. Combin. Theory (A), 71 (1995),\vspace{0.7mm} 175-218

\bibitem{dzi} \textsc{M. Dziemianczuk},  \textit{Generalizing Delannoy numbers via counting weighted lattice paths}. Integers 13 (2013) \vspace{0.7mm}A54

\bibitem{frr} \textsc{R. Fray, D. Roselle}, \textit{Weighted lattice paths},  \vspace{0.7mm} Pacific J. Math. 37.1 (1971) 85-96

\bibitem{het}\textsc{G. Hetyei}, \textit{Shifted Jacobi polynomials and Delannoy numbers}, arXiv:0909.5512 \vspace{0.7mm}(2009)

\bibitem{hog}\textsc{S. Hoggar}, \textit{Chromatic polynomials and logarithmic concavity}, J. Combin. Theory (B) 16  (1974) 248-\vspace{0.7mm}254 

\bibitem{hog1} \textsc{V. E. Hoggatt, Jr., M. Bicknell}, \textit{ Convolution triangles for generalized Fibonacci numbers}, The Fibonacci Quarterly 8(2) (1970) \vspace{0.7mm}158-171

\bibitem{kar} \textsc{S. Karlin}, \textit{Total Positivity}, vol. I, Stanford Univ. Press \vspace{0.7mm}(1968)

\bibitem {kur} \textsc{D. C. Kurtz}, \textit{A note on concavity properties of triangular arrays of numbers},  J. Combin. Theory Ser. A 13 (1972) \vspace{0.7mm} 135-139

\bibitem{lin}\textsc{Z. Lin, J. Zeng}, \textit{Positivity properties of Jacobi-Stirling numbers and generalized
Ramanujan polynomials}, Adv. in Appl. Math. 53 (2014) 12-27\vspace{0.7mm}
 
\bibitem{liu}\textsc{L. L. Liu, Y. Wang}, \textit{On the log-convexity of combinatorial sequences}, Adv. in Appl. Math. 39 (2007) 453-476. \vspace{0.7mm}

\bibitem{los}\textsc{N. A. Loehr, C. D. Savage}, \textit{Generalizing the combinatorics of binomial coefficients via l-nomials}, Integers \vspace{0.7mm}10.5 (2010) 531-558
 
\bibitem {maj} \textsc{L. Major},\ \textit{Unimodality and log-concavity of $f$-vectors for cyclic and ordinary polytopes}, Discrete Appl. Math.  161, 10-11,\vspace{0.7mm} (2013)  1669-1672.

\bibitem{men} \textsc{K. V. Menon}, \textit{On the convolution of logarithmically concave sequences,} Proc. Amer. Math. Soc. 23 \vspace{0.7mm} (1969) 439-441.

\bibitem{mez}\textsc{I. Mez\H{o}}, \textit{On the maximum of $r$-Stirling numbers}, Adv. in Appl. Math. 41.3 (2008) \vspace{0.7mm} 293-306.


\bibitem {muz} \textsc{L. Mu, S. Zheng}, \textit{On the Total Positivity of Delannoy-Like Triangles}  J. Integer Seq., 20.2 (2017) 3\vspace{0.7mm}.

\bibitem {sta} \textsc{R. P. Stanley}, \textit{Log-concave and unimodal sequences in algebra, combinatorics, and geometry, in Graph theory and its applications} :  Ann. New York Acad. Sci., 576 (1), 
 \vspace{0.7mm}(1989) 500-535. 
 
 \bibitem {suw} \textsc{X. T. Su, Y. Wang}, \textit{On unimodality problems in Pascal's triangle},  Electron. J. Combin. 15.1 (2008) \vspace{0.7mm} R113.
 
 \bibitem {suwa} \textsc{X. T. Su, Y. Wang}, \textit{Unimodality problems of multinomial coefficients and symmetric functions},  Electron. J. Combin. 18.1 (2011)\vspace{0.7mm} 73
 
 
 \bibitem {suz} \textsc{X. T. Su, W. W. Zhang}, \textit{Unimodal rays of the generalized Pascal's triangle},  ARS COMBINATORIA 108 (2013)\vspace{0.7mm} 289-296


\bibitem{yam}\textsc{Y. Yaming}, \textit{Confirming two conjectures of Su and Wang on binomial coefficients}, Adv. in Appl. Math. 43.4 (2009),\vspace{0.7mm} 317-322.

\bibitem{zhu}\textsc{B. X. Zhu}, \textit{Some positivities in certain triangular arrays}, Proc. Amer. Math. Soc. 142.9 (2014), 2943-2952.


\end{thebibliography}
\end{document}